\documentclass[11pt,reqno,letterpaper]{amsart}

\usepackage{booktabs}
\usepackage{color}
\usepackage[colorlinks=true, allcolors=blue,backref=page]{hyperref}
\usepackage{amsmath, amssymb, amsthm}
\usepackage{mathrsfs}
\usepackage{mathtools}
\usepackage[noabbrev,capitalize,nameinlink]{cleveref}
\usepackage[noadjust]{cite}
\usepackage{graphics}
\usepackage{pifont}
\usepackage{tikz}
\usepackage{bbm}
\usepackage{thm-restate}
\usepackage[T1]{fontenc}

\usetikzlibrary{arrows.meta}

\usepackage{environ}
\usepackage{framed}
\usepackage{url}
\usepackage[linesnumbered,ruled,vlined]{algorithm2e}
\usepackage[noend]{algpseudocode}
\usepackage[labelfont=bf]{caption}
\usepackage[framemethod=tikz]{mdframed}
\usepackage{appendix}
\usepackage{graphicx}
\usepackage[textsize=tiny]{todonotes}
\usepackage{tcolorbox}
\usepackage{enumerate}
\usepackage{stmaryrd}

\usepackage[margin=1in]{geometry}

\usepackage[shortlabels]{enumitem}
\crefformat{enumi}{#2#1#3}
\crefrangeformat{enumi}{#3#1#4 to~#5#2#6}
\crefmultiformat{enumi}{#2#1#3}
{ and~#2#1#3}{, #2#1#3}{ and~#2#1#3}

\DeclareSymbolFont{symbolsC}{U}{pxsyc}{m}{n}
\SetSymbolFont{symbolsC}{bold}{U}{pxsyc}{bx}{n}
\DeclareFontSubstitution{U}{pxsyc}{m}{n}
\DeclareMathSymbol{\medcircle}{\mathbin}{symbolsC}{7}

\crefname{equation}{}{} %remove ``Equation''
\AtBeginEnvironment{appendices}{\crefalias{section}{appendix}} %appendices

\usepackage[color,final]{showkeys} %add in 'final' into parameter to remove showkeys

\colorlet{refkey}{orange!20}
\colorlet{labelkey}{blue!30}

\crefname{algocf}{Algorithm}{Algorithms}

\renewcommand{\P}{\mathbb{P}}

\newcommand{\R}{\mathbb{R}}

\numberwithin{equation}{section}
\newtheorem{theorem}{Theorem}[section]
\newtheorem{proposition}[theorem]{Proposition}
\newtheorem{lemma}[theorem]{Lemma}

\crefname{claim}{Claim}{Claims}

\newtheorem{conjecture}[theorem]{Conjecture}
\newtheorem*{question*}{Question}

\theoremstyle{definition}
\newtheorem{definition}[theorem]{Definition}

\newtheorem*{definition*}{Definition}
\newtheorem{example}[theorem]{Example}

\newtheorem{remark}[theorem]{Remark}



\newcommand{\bs}{\boldsymbol}
\newcommand{\mb}{\mathbb}
\newcommand{\mbf}{\mathbf}

\newcommand{\mc}{\mathcal}

\newcommand{\mr}{\mathrm}

\newcommand{\on}{\operatorname}

\newcommand{\KL}{\operatorname{KL}}	
\newcommand{\E}{\mathbb{E}}	

\newcommand{\eps}{\varepsilon}

\let\originalleft\left
\let\originalright\right
\renewcommand{\left}{\mathopen{}\mathclose\bgroup\originalleft}
\renewcommand{\right}{\aftergroup\egroup\originalright}

\allowdisplaybreaks

\newcommand{\ignore}[1]{}

\title{Optimal girth-dependent bounds for the Bethe approximation of the permanent}

\author{Dingding Dong}
\address{Department of Mathematics, California Institute of Technology, Pasadena, CA 91125, USA}
\email{ddong124@caltech.edu}

\author{Vishesh Jain}
\address{Department of Mathematics, Statistics, and Computer Science,
University of Illinois Chicago, Chicago, IL 60607, USA}
\email{visheshj@uic.edu}

\begin{document}
\begin{abstract}
For an $n\times n$ nonnegative matrix $A$, the Bethe permanent, which is computable in deterministic polynomial time, satisfies the tight  universal comparison
\[\on{Bethe}(A) \leq \on{per}(A) \leq 2^{n/2}\on{Bethe}(A).\]
The lower bound, due to Gurvits, is attained on forests. The upper bound, due to Anari and Rezaei, is attained by the adjacency matrix of a disjoint union of $4$-cycles. 

Confirming a conjecture of Anari, we provide an optimal girth-dependent refinement of the above comparison. More precisely, we show that if the bipartite support graph of $A$ has girth at least an even integer $g \geq 4$, then
\[\on{Bethe}(A) \leq \on{per}(A) \leq 2^{2n/g}\on{Bethe}(A).\]
The upper bound is attained by the adjacency matrix of a disjoint union of $g$-cycles.
\end{abstract}

\maketitle

\section{Introduction}\label{sec:introduction}

For a positive integer $n$, write $[n]:=\{1,\ldots,n\}$ and let
$\on{Sym}(n)$ denote the set of permutations of $[n]$. For a
matrix $A \in \mb{R}^{n\times n}$, the permanent of $A$ is defined by
\[
  \on{per}(A)
  :=\sum_{\sigma\in\on{Sym}(n)}
       \prod_{i=1}^n A_{i,\sigma(i)}.
\]
Although this formula differs from the determinant only in the
absence of signs, the two quantities have fundamentally different
computational behavior. The determinant can be computed in polynomial
time, whereas Valiant~\cite{Valiant} proved that computing the permanent is
$\#\mathrm{P}$-hard even for $0$--$1$ matrices.

Let $G_A=(L\sqcup R,E)$ be the
bipartite graph in which $L$ and $R$ are disjoint copies of $[n]$ and $ ij \in E$ if and only if $A_{ij} \neq 0$. We call $G_A$ the support graph of $A$, assign the edge $ij$
weight $A_{ij}$, and write $A_e := A_{ij}$ for $e = ij$. Then
\[
  \on{per}(A)
  =\sum_{M\in\mc{M}(G_A)}\prod_{e\in M}A_e,
\]
where $\mc{M}(G_A)$ denotes the set of perfect matchings of $G_A$. When $A$ is nonnegative, the permanent is therefore the partition function of the
weighted perfect-matching model. In particular, when $A$ is a
$0$--$1$ matrix, every perfect matching has weight one, and
$\on{per}(A)$ counts the perfect matchings of $G_A$.

We henceforth restrict to nonnegative matrices and consider the
problem of approximating the permanent efficiently. Jerrum, Sinclair,
and Vigoda~\cite{JSV} settled the randomized version by providing a fully polynomial
randomized approximation scheme. Given $\eps,\delta>0$,
their algorithm returns a $(1+\eps)$-approximation to $\on{per}(A)$
with probability at least $1-\delta$, in time polynomial in the input
size, $1/\eps$, and $\log(1/\delta)$.

\subsection{Deterministic approximation of the permanent}

In contrast, no deterministic analogue of the Jerrum--Sinclair--Vigoda
result is known. In fact, every known general deterministic polynomial-time guarantee loses a factor exponential in $n$. Throughout, we measure multiplicative guarantees by the width of a certified interval,~i.e.,~a factor $C \geq 1$ guarantee consists of polynomial-time computable
quantities $L(A)$ and $U(A)$ satisfying
\[
  L(A)\leq\on{per}(A)\leq U(A)
  \qquad\text{and}\qquad
  U(A)\leq C L(A).
\]

The first nontrivial general deterministic guarantee was obtained by
Linial, Samorodnitsky, and Wigderson~\cite{LSW}. Combining matrix
scaling with the van der Waerden inequality, they obtained a factor
$e^n$ approximation. A sequence of refinements by
Gurvits and Samorodnitsky~\cite{Gurvits2006,Samorodnitsky,GS}
eventually lowered the factor to approximately $1.9022^n$. Anari and
Rezaei~\cite{AR} subsequently obtained the factor $2^{n/2}$ by providing a sharp universal comparison between the Bethe permanent and the permanent (see \cref{sec:bethe-perm}). Very recently, Anari~\cite{AnariBeyondBethe} and Narang and Perkins~\cite{NarangPerkins} improved the factor to $c^{n}$ for some absolute constant $c < \sqrt{2}$ by supplementing the Bethe approximation with additional efficiently computable corrections.

As observed by Anari and Rezaei, it is not a coincidence that the gap between randomized and deterministic approximations is $(1+\eps)$ vs. exponential. Indeed, by a simple tensorization trick, a deterministic polynomial-time algorithm which approximates the permanent of
every $n\times n$ nonnegative matrix within a factor $2^{n^c}$, for fixed $c < 1$, can be boosted to obtain a deterministic fully polynomial-time approximation scheme. 

\subsection{The Bethe permanent}
\label{sec:bethe-perm}
Anari and Rezaei obtained their $2^{n/2}$ bound by analyzing the
Bethe permanent, a heuristic approximation to the permanent originating in
statistical physics. For a bipartite graph $G$, let
\[
  \on{PM}(G)
  :=\left\{
      \bs{p}\in\R_{\geq0}^{E(G)}:
      \sum_{e\ni v}p_e=1
      \ \text{for every }v\in V(G)
    \right\}
\]
denote the fractional perfect matching polytope. If
$\on{PM}(G_A)=\emptyset$, we set $\on{Bethe}(A)=0$. If
$\on{PM}(G_A)\ne\emptyset$, define
\begin{equation}
\label{eq:bethe-product}
  \on{Bethe}(A)
  :=\max_{\bs{p}\in\on{PM}(G_A)}
      \prod_{e\in E(G_A)}
      \left(\frac{A_e}{p_e}\right)^{p_e}
      (1-p_e)^{1-p_e},
\end{equation}
where the boundary terms are defined by continuity. It will be useful to decompose the variational expression into an energy term and an entropy term. For a bipartite graph $G$ and
$\bs{p}\in\on{PM}(G)$, define the Bethe entropy by
\begin{equation}
  H_{\mr B}(\bs{p})
  :=
  \sum_{e\in E(G)}
  \bigl(
    -p_e\log p_e+(1-p_e)\log(1-p_e)
  \bigr).
\label{eq:bethe-entropy}
\end{equation}
Taking logarithms in \cref{eq:bethe-product}, we obtain
\begin{equation}
  \log\on{Bethe}(A)
  =
  \max_{\bs{p}\in\on{PM}(G_A)}
  \left\{
    \sum_{e\in E(G_A)}p_e\log A_e
    +H_{\mr B}(\bs{p})
  \right\}
\label{eq:logbethe}
\end{equation}
whenever $\on{PM}(G_A)\ne\emptyset$. Vontobel~\cite{Vontobel} proved that the Bethe entropy is concave on
$\on{PM}(G_A)$. Since the energy term is linear, the objective in \cref{eq:logbethe}
is concave, and maximizing it over the explicitly described polytope
$\on{PM}(G_A)$ is a convex optimization problem. In particular,
$\log\on{Bethe}(A)$ can be approximated to any prescribed additive
accuracy in polynomial time, or equivalently,
$\on{Bethe}(A)$ can be approximated to prescribed multiplicative
accuracy~\cite{Vontobel,GS} in polynomial time. 

Gurvits~\cite{Gurvits} proved that
\begin{equation}
\label{eq:gurvits-bound}
  \on{Bethe}(A)\leq\on{per}(A)
\end{equation}
for every nonnegative matrix $A$ and Gurvits and
Samorodnitsky~\cite{GS} obtained the reverse inequality with a multiplicative loss of $2^n$. Anari and Rezaei~\cite{AR} proved the sharp
universal bound
\begin{equation}
\label{eq:AR-bound}
  \on{per}(A)
  \leq 2^{n/2}\on{Bethe}(A).
\end{equation}

The lower and upper bounds are attained at opposite extremes. On tree-structured graphical models, the Bethe approximation is
exact~\cite{YedidiaFreemanWeiss}. For perfect matchings, this follows directly. A forest has at most one
perfect matching, since the symmetric difference of two perfect
matchings would contain a cycle. Moreover, if a perfect matching
exists, repeated deletion of leaves shows that its indicator vector
is the unique point of $\on{PM}(G_A)$. Hence
$
  \on{Bethe}(A)=\on{per}(A)
$
whenever $G_A$ is a forest. By contrast, equality in \cref{eq:AR-bound} is obtained by the bipartite adjacency matrix of a disjoint union of $4$-cycles, the shortest possible cycles in a bipartite graph. 

\subsection{Our contribution}

The contrast between forests and $4$-cycles suggests that the error
of the Bethe approximation should be controlled by the girth of the
support graph. Recall that the girth of a graph is the length of its
shortest cycle (which is necessarily even for a bipartite graph), with
forests regarded as having infinite girth. Motivated by this, Nima Anari
proposed the following conjecture in private communication.

\begin{conjecture}[Anari]
\label{conj:Anari}
Let $g\geq4$ be even, and let
$A\in\R_{\geq0}^{n\times n}$ be a nonnegative matrix whose support
graph has girth at least $g$. Then
\[
  \on{per}(A)
  \leq 2^{2n/g}\on{Bethe}(A).
\]
\end{conjecture}

\begin{example}
\label{ex:cycle-extremizers}
When $g\mid 2n$, the factor $2^{2n/g}$ in
\cref{conj:Anari} is attained by the bipartite adjacency matrix of a
disjoint union of $g$-cycles. Indeed, let $C_g$ be the
$(g/2)\times(g/2)$ $0$--$1$ matrix whose support graph is a single
cycle of length $g$. The cycle has exactly two perfect matchings, and
hence
$
  \on{per}(C_g)=2.
$
On the other hand, every fractional perfect matching of the cycle has
alternating edge weights
\[
  x,1-x,x,1-x,\ldots
\]
for some $x\in[0,1]$. For $x\in(0,1)$, the two factors in
\cref{eq:bethe-product} corresponding to a consecutive pair of edge
weights multiply to
\[
  \left(\frac1x\right)^x(1-x)^{1-x}
  \left(\frac1{1-x}\right)^{1-x}x^x
  =1,
\]
and the endpoint cases follow by continuity. Thus
$
  \on{Bethe}(C_g)=1.
$
Taking $A$ to be the block diagonal matrix consisting of $2n/g$
copies of $C_g$, the multiplicativity of both the permanent and the
Bethe permanent gives
\[
  \on{per}(A)
  =2^{2n/g}\on{Bethe}(A).
\]
\end{example}

Our main result confirms \cref{conj:Anari}.

\begin{theorem}
\label{thm:main}
Let $g\geq4$ be even, and let
$A\in\R_{\geq0}^{n\times n}$ be a nonnegative matrix whose support
graph has girth at least $g$. Then
\[
  \on{Bethe}(A)
  \leq\on{per}(A)
  \leq2^{2n/g}\on{Bethe}(A).
\]
Moreover, whenever $g\mid2n$, the upper bound is attained by the
$0$--$1$ matrices whose support graphs are disjoint unions of cycles
of length $g$.
\end{theorem}

When $g=4$, \cref{thm:main} recovers the sharp universal bound of
Anari and Rezaei. At the other extreme, the Bethe approximation is
exact on forests, as discussed above. More generally, along any
sequence of matrices whose support graphs have girth tending to infinity,
the Bethe permanent approximates the permanent within a
subexponential factor.

\subsection{Relation to previous work} As discussed above, the Bethe permanent is exact when the support
graph is a forest. This is consistent with the more general
exactness of the Bethe approximation on tree-structured graphical
models~\cite{YedidiaFreemanWeiss}. On graphs with cycles, loop
calculus gives exact formulas for the discrepancy between the true
partition function and its Bethe approximation. Chertkov and
Chernyak~\cite{ChertkovChernyak} expressed the partition function as the Bethe contribution
plus correction terms indexed by generalized loops,
and Watanabe and Chertkov~\cite{WatanabeChertkov} derived a version of this expansion for the
permanent. These formulas describe the
correction in terms of the individual loops and their weights. By
contrast, \cref{thm:main} gives a sharp worst-case bound depending
only on the girth of the support graph.

Several works show that the Bethe prediction becomes asymptotically exact for matching and related models on graph sequences that become locally tree-like; see, for example,
\cite{ZdeborovaMezard,BayatiNair,DemboMontanariSun,Meszaros}.
In the setting of perfect matchings, Ab\'ert, Csikv\'ari, Frenkel,
and Kun~\cite{ACFK} showed that for $d$-regular bipartite graphs converging locally to the
$d$-regular tree,  the normalized logarithm of the number of perfect
matchings converges to the Bethe prediction; see
also Lelarge~\cite{Lelarge} for related results. In the other direction, Csikv\'ari~\cite{CsikvariVertexTransitive} showed that for $d$-regular
vertex-transitive bipartite graphs, the presence of a cycle of
bounded length gives a strict exponential improvement over this tree
rate.
The former results concern locally tree-like limits
or finite-activity models, while Csikv\'ari's result requires
regularity and vertex transitivity. By contrast, \cref{thm:main}
gives a sharp non-asymptotic bound for weighted perfect matchings
without any additional assumptions. 

\subsection{Overview of the proof} Suppose that $\on{per}(A)>0$. Let $\mu_A$ be the Gibbs distribution
induced by $A$ on the perfect matchings of $G_A$, given by
\[
\mu_A(M):=\frac{\prod_{e\in M}A_e}{\on{per}(A)},
\]
and let $\bs{p}$ be its vector of edge marginals. Then
\begin{align*}
\log\on{per}(A)
&=\sum_{e\in E(G_A)}p_e\log A_e+H(\mu_A),\\
\log\on{Bethe}(A)
&\geq\sum_{e\in E(G_A)}p_e\log A_e+H_{\mr B}(\bs{p}).
\end{align*}
Thus, the common energy term cancels, and the problem reduces to
comparing $H(\mu_A)$ with $H_{\mr B}(\bs{p})$.

Anari and Rezaei control this discrepancy by a vertex-wise functional
$F$. Their argument can be used to show (see~\cite[Section~4, especially Equation~(6)]{AR}) that
\[
H(\mu_A)-H_{\mr B}(\bs{p})
\leq
\frac12\sum_{v\in V(G_A)}F(\bs{p}_v).
\]
Here $F$, defined in \cref{def:s-i-pi}, depends only on the local
marginal vector $\bs{p}_v=(p_e)_{e\ni v}$. Their sharp bound
$F(\bs{q})\leq(\log2)/2$, recorded in \cref{prop:AR}, then yields
their universal $2^{n/2}$ approximation.

Our main technical result, \cref{thm:entropy}, replaces $1/2$ by
$2/g$. Namely, every distribution $\mu$ on the perfect matchings of a
graph $G$ of girth at least $g$, with edge-marginal vector $\bs{p}$,
satisfies
\begin{equation}
\label{eq:intro-entropy-bound}
H(\mu)-H_{\mr B}(\bs{p})
\leq
\frac2g\sum_{v\in V(G)}F(\bs{p}_v).
\end{equation}
Taking $\mu=\mu_A$ and applying \cref{prop:AR} at the $2n$ vertices of
$G_A$ gives \cref{thm:main}.

\medskip

\paragraph{\bf Reduction to a good vertex} Write
$\Phi_G(\bs{p})
:=H_{\mr B}(\bs{p})+\frac2g\sum_vF(\bs{p}_v)$. Let $M\sim\mu$, and for $u\in L$ let $X_u$ be the vertex matched to $u$. 
We prove $H(\mu)\leq\Phi_G(\bs{p})$ by induction. Conditioning
on $X_u=w$ fixes the edge $uw$, and deleting $u$ and $w$ gives a
smaller perfect-matching instance whose girth is still at least $g$.
The chain rule for entropy and the induction hypothesis therefore
reduce the proof to finding $u\in L$ such that
\[
\Phi_G(\bs{p}) - \E_{X_u}\Phi_G(\bs{p}^{X_u})
\geq
H(X_u),
\]
where 
\[
  \bs{p}^{X_u}
  :=\bigl(\P(e\in M\mid X_u)\bigr)_{e\in E(G)}
\]
is the conditional marginal vector. We call such
a vertex {good}; see \cref{def:good-vertex}. 

\medskip

\paragraph{\bf Where girth enters}
Fix \(u\in L\), and for each edge \(e\in E(G)\) let $E_e := \mathbf 1_{\{e\in M\}}.$ 
We decompose the mutual information \(I(X_u;E_e)\) according to whether
\(E_e=1\) or \(E_e=0\). Summing the $E_e=1$ terms over $e\ni v$ gives
\[
H(\bs p_v)-\E_{X_u}H(\bs p_v^{X_u}),
\]
while summing the $E_e=0$ terms gives
\[
H_{\mr c}(\bs p_v)-\E_{X_u}H_{\mr c}(\bs p_v^{X_u}),
\]
where $H_{\mr c}$ is the complementary entropy functional defined in
\cref{def:entropy}.

By the perfect matching constraint, if an edge $e$ incident to $v$ is absent, then
exactly one of the other edges $f\neq e$ incident to $v$ is present. Convexity
of relative entropy therefore bounds the contribution from $E_e=0$ by
the contributions from $(E_f = 1)_{f\neq e}$. In
\cref{sec:constructing-flow}, we use this comparison to construct an information flow on an auxiliary network. After decomposing the flow into paths and cycles, each
component projects to a nonbacktracking walk in \(G\). 

By construction, the steps leaving $u$ account for the $E_e=1$
contributions at $u$, whose total weight is $H(X_u)$. Every departure
after the first requires the walk to return to $u$ before it can leave
again. The resulting excursion is a closed nonbacktracking walk, and
hence has length at least $g$. A weighted count of these excursions,
carried out in \cref{prop:girth-numerical}, gives the following estimate. Define
$$
\Phi_0(\bs p)
:=
\frac1g\sum_{v\in V}H_{\mr c}(\bs p_v)
+\frac{g-2}{2g}\sum_{v\in V}S(\bs p_v),
$$
where $S:=H-H_{\mr c}$ is the local Bethe entropy functional from
\cref{def:entropy}. Then \cref{prop:girth-estimate} states that
\[
\Phi_0(\bs p)-\E_{X_u}\Phi_0(\bs p^{X_u})-H(X_u)
\geq
-\frac2gS(\bs p_u).
\]
Thus the girth argument controls $\Phi_0$, rather than the full
potential $\Phi$, and leaves a deficit of $2S(\bs p_u)/g$. It remains
to control $\Phi-\Phi_0$.

\medskip

\paragraph{\bf Averaging}
A direct calculation gives $$ \Phi(\bs p)-\Phi_0(\bs p) = \frac1g\sum_{v\in V}\Gamma(\bs p_v), \qquad \Gamma(\bs q):= H(\bs q)-2H_{\mr c}(\bs q)+2F(\bs q). $$ 
For a fixed vertex $u$, the girth estimate would therefore imply that $u$ is good if $$ \sum_{v\in V} \left( \Gamma(\bs p_v) -\E_{X_u}\Gamma(\bs p_v^{X_u}) \right) \geq 2S(\bs p_u). $$
This pointwise inequality is false in general. Instead, we prove the averaged version
$$
\sum_{u\in L}\sum_{v\in V}
\left(
\Gamma(\bs p_v)-\E_{X_u}\Gamma(\bs p_v^{X_u})
\right)
\geq
2\sum_{u\in L}S(\bs p_u).
$$
Together with the girth estimate, this guarantees that
some vertex is good.

Averaging allows us to interchange the order of summation:
$$
\sum_{u\in L}\sum_{v\in V}
\left(
\Gamma(\bs p_v)-\E_{X_u}\Gamma(\bs p_v^{X_u})
\right)
=
\sum_{v\in V}\sum_{u\in L}
\left(
\Gamma(\bs p_v)-\E_{X_u}\Gamma(\bs p_v^{X_u})
\right).
$$
Thus it suffices to control, for each $v\in V$, the total change in $\Gamma(\bs p_v)$ obtained by conditioning on the coordinates $X_u$, $u\in L$. Since $X$ is a bijection, if $v\in L$, conditioning on $X_v$ determines the matched neighbor of $v$, while conditioning on $X_u$ for $u\neq v$ rules out one possible matched neighbor. The key recursive identity for $\Gamma$ (\cref{lem:delete-one-identity}) then shows that
$$
\sum_{u\in L}
\left(
\Gamma(\bs p_v)-\E_{X_u}\Gamma(\bs p_v^{X_u})
\right)
\geq S(\bs p_v).
$$
The corresponding argument for $X^{-1}$ treats $v \in R$. Summing over $v \in V$ gives
$$
\sum_{u\in L}\sum_{v\in V}
\left(
\Gamma(\bs p_v)-\E_{X_u}\Gamma(\bs p_v^{X_u})
\right)
\geq
\sum_{v\in V}S(\bs p_v)
=
2\sum_{u\in L}S(\bs p_u).
$$
Thus the averaged contribution from $\Phi-\Phi_0$ exactly compensates
for the total deficit left by the girth estimate. 

\subsection{Acknowledgements} The authors thank Nima Anari, Clayton Mizgerd, and Huy Tuan Pham for helpful discussions. V.J.~is supported by NSF grant DMS-2237646. This work was initiated when D.D.~visited V.J.~at the University of Illinois Chicago. 

\medskip

\paragraph{\bf Statement on AI use} The authors had developed a strategy for proving \cref{thm:main} with multiplicative factor $\exp(n/f(g))$, where $f(g) \to \infty$ as $g \to \infty$. ChatGPT 5.6 Sol Pro was able to develop a version of this strategy into a proof with $f(g) = \Theta(g/\log g)$. Subsequent interactions with ChatGPT 5.6 Sol Ultra, aimed at understanding the source of this logarithmic loss and the relationship with the work of Anari--Rezaei, led to the development of the present proof. The authors used Codex for assistance with preparing the manuscript. The mathematical content, the final text, and any errors are the responsibility of the authors.

\section{Preliminaries}\label{sec:preliminaries}

\subsection{Bethe entropy and its local decomposition}

All logarithms are natural and we use the convention $0\log 0 = 0$ which is imposed by continuity. For $d\geq1$, let
\[
   \Delta_d:=\left\{\bs{q}\in\R_{\geq0}^d:
                     \sum_{i=1}^d q_i=1\right\}
\]
denote the probability simplex.

\begin{definition}\label{def:entropy}
For $\bs{q}=(q_1,\ldots,q_d)\in\Delta_d$, let
\[
  H(\bs{q}):=-\sum_{i=1}^d q_i\log q_i
\]
denote the usual Shannon entropy, and define the complementary entropy functional
\[
  H_{\mr c}(\bs{q}):=-\sum_{i=1}^d(1-q_i)\log(1-q_i).
\]
We call the difference
\[
  S(\bs{q}):=H(\bs{q})-H_{\mr c}(\bs{q})
\]
the local Bethe entropy functional.
If $\mu$ is a
probability distribution on a finite set, we write $H(\mu)$ for the
entropy of its probability vector; for a finite random variable $U$, we write $\mc{L}(U)$ for its law and set $H(U):=H(\mc{L}(U))$.
\end{definition}

\begin{remark}
    The terminology complementary entropy functional is motivated by the following. Writing
\[
  h(t):=-t\log t-(1-t)\log(1-t)
\]
for the binary entropy function, we have
$
  H(\bs{q})+H_{\mr c}(\bs{q})
  =\sum_{i=1}^d h(q_i).
$
We emphasize that $H_{\mr c}$ is not generally the entropy of a probability
distribution, since
$
  \sum_{i=1}^d(1-q_i)=d-1.
$
\end{remark}

The functional $S$ is the local building block of the Bethe
entropy. Let $G=(V,E)$ be a bipartite graph and let
$\bs{q}\in\on{PM}(G)$. For each $v\in V$, write
$
  \bs{q}_v:=(q_e)_{e\ni v}\in\Delta_{d_v},
$
where $d_v$ is the degree of $v$. In this notation, the Bethe entropy
is
\begin{align}\label{eq:bethe-entropy-local}
  H_{\mr B}(\bs{q})
  &:=\sum_{e\in E}
      \bigl(-q_e\log q_e+(1-q_e)\log(1-q_e)\bigr) =\frac12\sum_{v\in V}S(\bs{q}_v).
\end{align}
The edgewise expression is due to Vontobel
\cite[Corollary~15]{Vontobel}. The functional $S$ appears in
\cite[Definition~19]{Vontobel}, and the corresponding local
decomposition of the Bethe entropy is given in
\cite[Lemma~21]{Vontobel}. The identity in
\cref{eq:bethe-entropy-local} also explains our terminology ``local Bethe entropy functional'', as each
$S(\bs{q}_v)$ is the contribution at the vertex $v$ to the global Bethe entropy. 

We will repeatedly use the following two properties of $S$.

\begin{lemma}\label{lem:b-properties}
For every $d\geq1$, the functional $S$ is nonnegative and concave on
$\Delta_d$.
\end{lemma}

\begin{proof}
This is \cite[Theorem~20]{Vontobel}.
\end{proof}

% The Bethe entropy gives the following variational representation of
% the Bethe permanent.

\begin{proposition}
\label{prop:bethe-variational}
For every nonnegative matrix $A$ such that
$\on{PM}(G_A)\ne\emptyset$,
\[
  \log\on{Bethe}(A)
  =
  \max_{\bs{q}\in\on{PM}(G_A)}
  \left\{
    \sum_{e\in E(G_A)}q_e\log A_e
    +H_{\mr B}(\bs{q})
  \right\}.
\]
\end{proposition}

\begin{proof}
This follows by taking logarithms in
\cref{eq:bethe-product}; see also \cite[Corollary~15]{Vontobel}.
\end{proof}

\subsection{KL divergence and information gain}

For a finite set $\Omega$, let $\Delta(\Omega)$ denote the set of
probability distributions on $\Omega$. Thus, under the natural
identification, $\Delta_d=\Delta([d])$.

\begin{definition}\label{def:kl-divergence}
Let $\mu,\nu\in\Delta(\Omega)$. The Kullback--Leibler divergence,
also called the relative entropy, of $\mu$ relative to $\nu$ is
\[
  \KL(\mu\|\nu)
  :=\sum_{\omega\in\Omega}
     \mu(\omega)\log\frac{\mu(\omega)}{\nu(\omega)},
\]
with the conventions
$0\log\frac0y=0$ for $y\geq 0$ and $x\log\frac x0=+\infty$ for $x>0$.

\end{definition}

We next recall the connection between KL divergence and mutual
information. For finite random variables $U$ and $Y$, their mutual information is
\begin{align}\label{eq:mutual-information}
  I(U;Y)
  &:=H(U)-\E_Y H\bigl(\mc{L}(U\mid Y)\bigr)
  =\E_Y\KL\bigl(
       \mc{L}(U\mid Y)\,\big\|\,\mc{L}(U)
     \bigr).
\end{align}
Thus,
mutual information can be viewed either as the reduction in the
entropy of $U$ when $Y$ is observed or as the expected divergence of
the posterior law of $U$ from its prior law. The first viewpoint
extends naturally to functionals other than entropy.

\begin{definition}\label{def:Psi-information}
Let $U$ and $Y$ be finite random variables, with $U$ taking values in
$\Omega$, and let $\Psi:\Delta(\Omega)\to\R$. The
{$\Psi$-information gain} from observing $Y$ is
\begin{equation}\label{eq:Psi-information}
  \mathsf I_\Psi(U;Y)
  :=\Psi(\mc{L}(U))
    -\E_Y\Psi(\mc{L}(U\mid Y)).
\end{equation}
In particular, $\mathsf I_H(U;Y)=I(U;Y)$.
\end{definition}

Unlike ordinary mutual information, $\Psi$-information gain need not
be symmetric in its two arguments. Nevertheless, when $\Psi$ is
concave, $\Psi$-information gain is nonnegative and satisfies a
data-processing inequality under further randomization of the
observed variable. We record these properties next.

\begin{lemma}
\label{lem:data-processing}
Let $U,Y,Z$ be finite random variables, with $U$ taking values in
$\Omega$, and let $\Psi:\Delta(\Omega)\to\R$ be concave. Suppose that
$U$--$Y$--$Z$ is a Markov chain; equivalently, $U$ and $Z$ are conditionally
independent given $Y$. Then
\[
  \mathsf I_\Psi(U;Y)
  \geq \mathsf I_\Psi(U;Z)
  \geq 0.
\]
\end{lemma}

\begin{proof}
We first note that $\Psi$-information gain is nonnegative. Indeed, for any finite random variable $W$,
\[
  \mc{L}(U)=\E_W\mc{L}(U\mid W).
\]
Concavity of $\Psi$ and Jensen's inequality therefore give
\[
  \Psi(\mc{L}(U))
  \geq \E_W\Psi(\mc{L}(U\mid W)),
\]
and hence $\mathsf I_\Psi(U;W)\geq0$.

It remains to prove data processing. The Markov-chain assumption
implies
\[
  \mc{L}(U\mid Z)
  =\E\bigl[\mc{L}(U\mid Y)\mid Z\bigr].
\]
Indeed, this identity holds coordinatewise because, for every
$\omega\in\Omega$,
\[
  \P(U=\omega\mid Z)
  =\E\bigl[\P(U=\omega\mid Y)\mid Z\bigr].
\]
Applying Jensen's inequality gives
\[
  \Psi(\mc{L}(U\mid Z))
  \geq
  \E\bigl[\Psi(\mc{L}(U\mid Y))\mid Z\bigr].
\]
Taking expectations and subtracting from $\Psi(\mc{L}(U))$ yields
$
  \mathsf I_\Psi(U;Y)\geq\mathsf I_\Psi(U;Z),
$
as claimed.
\end{proof}

\subsection{The Anari--Rezaei functional}

We next introduce the local functional underlying the sharp universal
bound of Anari and Rezaei~\cite{AR}. We regard a permutation
$\pi\in\on{Sym}(d)$ as the ordered list
$
  \pi(1),\ldots,\pi(d).
$

\begin{definition}\label{def:s-i-pi}
For $\bs{q}\in\Delta_d$, $\pi\in\on{Sym}(d)$, and $t\in[d]$, define
the {suffix mass}
\[
  s_t^\pi(\bs{q})
  :=\sum_{r=t}^d q_{\pi(r)}.
\]
Thus, $s_t^\pi(\bs{q})$ is the total $\bs{q}$-mass of the elements in
position $t$ or later in the ordering $\pi$. Define
\begin{align*}
  F(\bs{q})
  &:=\E_\pi\left[
       \sum_{t=1}^d q_{\pi(t)}
       \log\frac{s_t^\pi(\bs{q})}{q_{\pi(t)}}
     \right]
     -S(\bs{q})
  =\E_\pi\left[
       \sum_{t=1}^d q_{\pi(t)}
       \log s_t^\pi(\bs{q})
     \right]
     +H_{\mr c}(\bs{q}),
\end{align*}
where $\pi$ is uniformly distributed on $\on{Sym}(d)$. 
\end{definition}

\begin{remark}
\label{rem:zero-coordinates}
The functionals $H$, $H_{\mr c}$, $S$, and $F$ are unchanged by
inserting or deleting zero coordinates, and they vanish at point
masses. For $H$, $H_{\mr c}$, and $S$, this follows immediately from
their definitions. For $F$, a zero coordinate contributes nothing,
does not change any suffix mass, and the relative order induced on
the remaining coordinates is still uniform.
\end{remark}

In the work of Anari and Rezaei, the authors use a uniformly random ordering to
define an auxiliary comparison distribution on perfect matchings, and
the suffix masses $s_t^\pi(\bs{q})$ appear as its local normalizing
factors. Comparing the resulting local entropy term with
$S(\bs{q})$ produces $F(\bs{q})$. We need the following sharp
estimate.

\begin{proposition}[{\cite[Lemma~10]{AR}}]\label{prop:AR}
For every $d\geq1$ and every $\bs{q}\in\Delta_d$,
\[
  F(\bs{q})\leq\frac{\log2}{2}.
\]
\end{proposition}
Equality is attained, for example, at
$\bs{q}=(1/2,1/2)\in\Delta_2$. 

\section{Reduction to a good vertex}
\label{sec:good-vertex-reduction}

We begin by isolating the entropy inequality underlying the proof of
\cref{thm:main}. Fix an even integer $g\geq4$. For a bipartite graph
$G=(V,E)$ of girth at least $g$ and $\bs{p}\in\on{PM}(G)$, define
\begin{equation}\label{eq:Phi-def}
  \Phi_G(\bs{p})
  :=H_{\mr B}(\bs{p})
    +\frac2g\sum_{v\in V}F(\bs{p}_v)
  =\frac12\sum_{v\in V}S(\bs{p}_v)
    +\frac2g\sum_{v\in V}F(\bs{p}_v).
\end{equation}
We suppress the subscript $G$ when the graph is clear from context.

\begin{theorem}\label{thm:entropy}
Let $G=(L\sqcup R,E)$ be a bipartite graph of girth at least $g$, with
$|L|=|R|=n$. Let $\mu$ be a probability distribution on the perfect
matchings of $G$, and let $
  p_e:=\P_{M\sim\mu}(e\in M)$, $\bs{p}:=(p_e)_{e\in E}$.
Then
\[
  H(\mu)\leq\Phi_G(\bs{p}).
\]
Equivalently,
\[
  H(\mu)-H_{\mr B}(\bs{p})
  \leq\frac2g\sum_{v\in V}F(\bs{p}_v).
\]
\end{theorem}

We first show that this entropy inequality implies the main theorem. The lower bound in \cref{thm:main} is Gurvits's inequality
\cite{Gurvits}, while the sharpness statement follows from the cycle
construction in \cref{ex:cycle-extremizers}. It therefore remains to
derive the upper bound from \cref{thm:entropy}.

\begin{proof}[Proof of the upper bound in \cref{thm:main} assuming
\cref{thm:entropy}]
Let $A\in\R_{\geq0}^{n\times n}$ have support graph
$G_A=(V,E)$ of girth at least $g$. If
$\on{PM}(G_A)=\emptyset$, then
$
  \on{per}(A)=\on{Bethe}(A)=0,
$
and there is nothing to prove. We may therefore assume that
$\on{PM}(G_A)\ne\emptyset$. Since the bipartite perfect-matching polytope is integral, $G_A$ has a perfect matching in this case. Every edge of
$G_A$ has positive weight, and hence $\on{per}(A)>0$.

Consider the probability distribution $\mu$ on the set of perfect matchings $\mc{M}(G_A)$ given by
\[
  \mu(M):=\frac{\prod_{e\in M}A_e}{\on{per}(A)},
\]
and let $p_e:=\P_{M\sim\mu}(e\in M)$,  $\bs{p}:=(p_e)_{e\in E}$. Then $\bs{p}\in\on{PM}(G_A)$ and
\begin{align*}
  H(\mu)
  &=-\E_{M\sim\mu}\log\mu(M)
  =\log\on{per}(A)-\sum_{e\in E}p_e\log A_e.
\end{align*}
It follows from \cref{thm:entropy} that
\begin{align*}
  \log\on{per}(A)
  &\leq
    \sum_{e\in E}p_e\log A_e
    +H_{\mr B}(\bs{p})
    +\frac2g\sum_{v\in V}F(\bs{p}_v)\\
  &\leq
    \log\on{Bethe}(A)
    +\frac2g\cdot 2n\cdot\frac{\log2}{2}
    &&\text{(\cref{prop:bethe-variational,prop:AR})}\\
  &=\log\on{Bethe}(A)+\frac{2n\log2}{g},
\end{align*}
which, after exponentiating, gives the upper bound.
\end{proof}

We now reduce \cref{thm:entropy} to the existence of a ``good'' vertex. For
the remainder of the paper, fix a graph $G=(L\sqcup R,E)$ and a
distribution $\mu$ as in \cref{thm:entropy}. As before, let
\begin{equation}\label{eq:p-defns}
  p_e:=\P_{M \sim \mu}(e\in M),
  \qquad
  \bs{p}:=(p_e)_{e\in E}.
\end{equation}

For each $u\in L$, let $X_u$ be the unique vertex $w\in R$ such that
$uw\in M$. Thus
$
  X=(X_u)_{u\in L}
$
is a random bijection from $L$ to $R$. Also let
\begin{equation}\label{eq:p-defns-cond}
  p_e^{X_u}:=\P_{M \sim \mu}(e\in M\mid X_u),
  \qquad
  \bs{p}^{X_u}:=(p_e^{X_u})_{e\in E}.
\end{equation}

\begin{definition}\label{def:good-vertex}
For $u\in L$, define
\begin{equation}\label{eq:gap-def}
  \on{gap}(u)
  :=\Phi(\bs{p})
    -\E_{X_u}\Phi(\bs{p}^{X_u})
    -H(X_u).
\end{equation}
We call $u$ a {good vertex} if $\on{gap}(u)\geq0$.
\end{definition}

\cref{thm:entropy} follows from the following proposition, which says that a good vertex always exists. In fact, we prove a stronger averaged statement. 
\begin{proposition}
\label{prop:good-vertex}
Let $g\geq4$ be even, and let $G=(L\sqcup R,E)$ be a bipartite
graph of girth at least $g$, with $|L|=|R|=n\geq1$. Let $\mu$ be a
probability distribution on the perfect matchings of $G$, and define
$\bs{p}$, $\bs{p}^{X_u}$, and $\on{gap}(u)$ as in
\cref{eq:p-defns,eq:p-defns-cond,eq:gap-def}. Then
\[
  \sum_{u\in L}\on{gap}(u)\geq0.
\]
In particular, there exists a good vertex $u\in L$.
\end{proposition}

\begin{proof}[Proof of \cref{thm:entropy} assuming
\cref{prop:good-vertex}]
We proceed by induction on $n:=|L|=|R|$. The case $n\leq1$ is
immediate, so assume that $n > 1$. By \cref{prop:good-vertex}, choose $u\in L$ such that
$\on{gap}(u)\geq0$, or equivalently,
\[
  \Phi(\bs{p})
  \geq H(X_u)+\E_{X_u}\Phi(\bs{p}^{X_u}).
\]

For every $w\in R$ with $\P(X_u=w)>0$, conditioning on $X_u=w$
fixes the edge $uw$. Deleting its endpoints gives an instance of size $n-1$ without
changing the potential, by \cref{rem:zero-coordinates}. The induction hypothesis
therefore gives
\[
  H(\mu\mid X_u=w)\leq\Phi(\bs{p}^{X_u=w}).
\]
Therefore,
\[
  H(\mu)
  =H(X_u)+\E_{X_u}H(\mu\mid X_u)
  \leq H(X_u)+\E_{X_u}\Phi(\bs{p}^{X_u})
  =\Phi(\bs{p})-\on{gap}(u)
  \leq\Phi(\bs{p}). \qedhere
\]
\end{proof}

\subsection{The two key estimates}

It remains to prove \cref{prop:good-vertex}. The proof combines two
estimates. The first is a pointwise inequality and is the only place
where the girth hypothesis is used. The second is an averaged
inequality that uses only that $X$ is a bijection. We state the two
estimates here and prove them in \cref{sec:girth,sec:permutation},
respectively.

Define
\begin{equation}\label{eq:Phi0-def}
  \Phi_0(\bs{p})
  :=\frac1g\sum_{v\in V}H_{\mr c}(\bs{p}_v)
    +\frac{g-2}{2g}\sum_{v\in V}S(\bs{p}_v).
\end{equation}

\begin{proposition}[Girth estimate]\label{prop:girth-estimate}
For every $u\in L$,
\begin{equation*}
  \Phi_0(\bs{p})
  -\E_{X_u}\Phi_0(\bs{p}^{X_u})
  -H(X_u)
  \geq-\frac2gS(\bs{p}_u).
\end{equation*}
\end{proposition}

Subtracting \cref{eq:Phi0-def} from \cref{eq:Phi-def} and using
$S=H-H_{\mr c}$ gives
\begin{equation}\label{eq:Phi-decomposition}
  \Phi(\bs{p})-\Phi_0(\bs{p})
  =\frac1g\sum_{v\in V}\Gamma(\bs{p}_v),
  \qquad
  \Gamma(\bs{q})
  :=H(\bs{q})-2H_{\mr c}(\bs{q})+2F(\bs{q}).
\end{equation}

For each $v\in V$, let $Z_v$ be the vertex matched to $v$ by $M$.
Thus $Z_v=X_v$ when $v\in L$ and
$Z_v=(X^{-1})_v$ when $v\in R$. Identifying the possible values of
$Z_v$ with the edges incident to $v$, we have
\[
  \mc{L}(Z_v)=\bs{p}_v,
  \qquad
  \mc{L}(Z_v\mid X_u)=\bs{p}_v^{X_u}.
\]

\begin{proposition}[Permutation estimate]
\label{prop:averaging-estimate}
We have
\[
  \sum_{u\in L}\sum_{v\in V}
    \mathsf I_\Gamma(Z_v;X_u)
  \geq2\sum_{u\in L}S(\bs{p}_u).
\]
\end{proposition}

\begin{proof}[Proof of \cref{prop:good-vertex}]
By \cref{eq:Phi-decomposition} and the definition of information
gain,
\begin{align*}
  \sum_{u\in L}\on{gap}(u)
  &=
    \sum_{u\in L}
    \left(
      \Phi_0(\bs{p})
      -\E_{X_u}\Phi_0(\bs{p}^{X_u})
      -H(X_u)
    \right) + \frac1g\sum_{u\in L}\sum_{v\in V}
      \mathsf I_\Gamma(Z_v;X_u)\\
  &\geq
    -\frac2g\sum_{u\in L}S(\bs{p}_u)
    +\frac1g\sum_{u\in L}\sum_{v\in V}
      \mathsf I_\Gamma(Z_v;X_u)
    \\
  &\geq0;
\end{align*}
here, the first inequality uses \cref{prop:girth-estimate} and the second uses \cref{prop:averaging-estimate}.
\end{proof}

\section{The girth estimate}
\label{sec:girth}

We now prove \cref{prop:girth-estimate}. Fix an arbitrary vertex
$\star\in L$. We first convert the information revealed by
$X_\star$ into a flow on an auxiliary network. We then decompose this
flow and use the girth assumption on the resulting nonbacktracking
walks.

\subsection{Constructing the flow}
\label{sec:constructing-flow}

For each edge $e\in E$, let
\[
  E_e:=\mbf{1}_{\{e\in M\}}.
\]
The mutual information $I(X_\star;E_e)$ naturally splits according to
the two possible values of $E_e$.

\begin{definition}
\label{def:edge-information}
For every $e\in E$, define
\begin{align*}
  \mathsf J_e^+
  &:=
  p_e\KL\bigl(
    \mc{L}(X_\star\mid E_e=1)
    \,\big\|\,
    \mc{L}(X_\star)
  \bigr),\\
  \mathsf J_e^-
  &:=
  (1-p_e)\KL\bigl(
    \mc{L}(X_\star\mid E_e=0)
    \,\big\|\,
    \mc{L}(X_\star)
  \bigr).
\end{align*}
We set $\mathsf J_e^+=0$ when $p_e=0$ and
$\mathsf J_e^-=0$ when $p_e=1$; the formulas below are interpreted
using the same convention.
Thus, recalling \eqref{eq:mutual-information}, we have
\[
  I(X_\star;E_e)=\mathsf J_e^++\mathsf J_e^-.
\]
For every $v\in V$, define
\[
  \mathsf J_v^+
  :=\sum_{e\ni v}\mathsf J_e^+,
  \qquad
  \mathsf J_v^-
  :=\sum_{e\ni v}\mathsf J_e^-.
\]
\end{definition}

For brevity, write
\[
  p_e^{X_\star}:=\P(e\in M\mid X_\star).
\]
Bayes' rule gives
\begin{align*}
  \mathsf J_e^+
  &=
  \E_{X_\star}\left[
    p_e^{X_\star}\log\frac{p_e^{X_\star}}{p_e}
  \right],
  &
  \mathsf J_e^-
  &=
  \E_{X_\star}\left[
    (1-p_e^{X_\star})
    \log\frac{1-p_e^{X_\star}}{1-p_e}
  \right].
\end{align*}
Summing over the edges incident to $v$ gives
\begin{equation}
\label{eq:Jv-information}
\begin{aligned}
  \mathsf J_v^+
  &=
  H(\bs{p}_v)
  -\E_{X_\star}H(\bs{p}_v^{X_\star}),\\
  \mathsf J_v^-
  &=
  H_{\mr c}(\bs{p}_v)
  -\E_{X_\star}H_{\mr c}(\bs{p}_v^{X_\star}),\\
  \mathsf J_v^+-\mathsf J_v^-
  &=
  S(\bs{p}_v)
  -\E_{X_\star}S(\bs{p}_v^{X_\star}).
\end{aligned}
\end{equation}

We first record two inequalities satisfied by these quantities.

\begin{lemma}
\label{lem:J-bound}
For every $v\in V$ and every edge $e\ni v$,
\[
  \mathsf J_v^-\leq\mathsf J_v^+,
  \qquad
  \mathsf J_e^-
  \leq
  \sum_{\substack{f\ni v\\f\ne e}}\mathsf J_f^+.
\]
\end{lemma}

\begin{proof}
By \cref{eq:Jv-information},
\[
  \mathsf J_v^+-\mathsf J_v^-
  =
  S(\bs{p}_v)
  -\E_{X_\star}S(\bs{p}_v^{X_\star})
  \geq0,
\]
where the inequality follows from the concavity of $S$ and the
identity
\[
  \bs{p}_v=\E_{X_\star}\bs{p}_v^{X_\star};
\]
see \cref{lem:b-properties}.

For the second inequality, there is nothing to prove if $p_e=1$.
Suppose that $p_e<1$. Since exactly one edge incident to $v$ belongs
to $M$,
\[
  \mc{L}(X_\star\mid E_e=0)
  =
  \sum_{\substack{f\ni v\\f\ne e}}
  \frac{p_f}{1-p_e}\,
  \mc{L}(X_\star\mid E_f=1).
\]
Convexity of KL divergence in its first argument
\cite[Theorem~2.7.2]{CoverThomas} therefore gives
\[
  \mathsf J_e^-
  =
  (1-p_e)
  \KL\bigl(
    \mc{L}(X_\star\mid E_e=0)
    \,\big\|\,
    \mc{L}(X_\star)
  \bigr)
  \leq
  \sum_{\substack{f\ni v\\f\ne e}}
  p_f
  \KL\bigl(
    \mc{L}(X_\star\mid E_f=1)
    \,\big\|\,
    \mc{L}(X_\star)
  \bigr)
  =
  \sum_{\substack{f\ni v\\f\ne e}}\mathsf J_f^+.
  \qedhere
\]
\end{proof}

The flow construction and the girth argument depend on the quantities
$\mathsf J_e^\pm$ only through the inequalities in
\cref{lem:J-bound}. Accordingly, the remainder of the construction
uses only that the nonnegative weights $\mathsf J_e^\pm$ satisfy the
inequalities in \cref{lem:J-bound}.

Before constructing the auxiliary flow, we decompose these weights
locally. Fix $v\in V$. We seek nonnegative coefficients
$\alpha^v_{ef}$, indexed by distinct edges $e,f\ni v$, such that
\begin{equation}
\label{eq:local-decomposition}
\sum_{\substack{f\ni v\\f\ne e}}\alpha^v_{ef}
=\mathsf J_e^-,
\qquad
\sum_{\substack{e\ni v\\e\ne f}}\alpha^v_{ef}
\leq\mathsf J_f^+.
\end{equation}
Thus, the first identity distributes the entire weight
$\mathsf J_e^-$ among the other edges incident to $v$, while the
second inequality says that the total amount assigned to an edge $f$
is at most $\mathsf J_f^+$.

\begin{lemma}
\label{lem:local-decomposition}
For every $v\in V$, there exist nonnegative coefficients
$\alpha^v_{ef}$ satisfying \cref{eq:local-decomposition}.
\end{lemma}

\begin{proof}
Fix $v\in V$. Form a directed network with source $\mathbf s_v$,
sink $\mathbf t_v$, and two vertices $e^-$ and $e^+$ for each edge
$e\ni v$. Add the arcs
\[
\mathbf s_v\longrightarrow e^-
\quad\text{with capacity }\mathsf J_e^-,
\qquad
f^+\longrightarrow\mathbf t_v
\quad\text{with capacity }\mathsf J_f^+,
\]
and, whenever $e\ne f$, add an arc
\[
e^-\longrightarrow f^+ \quad \text{with capacity } \mathsf J_v^-.
\]
We claim that this network admits a flow of value $\mathsf J_v^-$.
By the max-flow--min-cut theorem, it suffices to verify that every
cut has capacity at least $\mathsf J_v^-$. Any cut containing an arc $e^-\to f^+$ already has this capacity. For any remaining cut, let $A$ be the set of vertices of the form
$e^-$ on the source side, and let $N(A)$ denote their neighbors among
the vertices of the form $f^+$. Since every vertex in $N(A)$ must
also lie on the source side, it is enough to show that
\[
\sum_{e^-\in A}\mathsf J_e^-
\leq
\sum_{f^+\in N(A)}\mathsf J_f^+.
\]
If $A=\{e^-\}$, this is the second inequality in
\cref{lem:J-bound}. If $|A|\geq2$, then $N(A)$ contains every vertex
$f^+$, and the required inequality follows from
\[
\sum_{e^-\in A}\mathsf J_e^-
\leq\mathsf J_v^-
\leq\mathsf J_v^+
=
\sum_{f\ni v}\mathsf J_f^+.
\]
The source cut has capacity
\[
\sum_{e\ni v}\mathsf J_e^-=\mathsf J_v^-.
\]
Thus the minimum cut, and hence the maximum flow, has value
$\mathsf J_v^-$.

Let $\alpha^v_{ef}$ be the flow on the arc
$e^-\to f^+$. Since every arc leaving $\mathbf s_v$ is saturated,
\[
\sum_{\substack{f\ni v\\f\ne e}}\alpha^v_{ef}
=\mathsf J_e^-.
\]
The capacities of the arcs entering $\mathbf t_v$ give
\[
\sum_{\substack{e\ni v\\e\ne f}}\alpha^v_{ef}
\leq\mathsf J_f^+.
\]
Hence \cref{eq:local-decomposition} holds.
\end{proof}

Choose coefficients as in \cref{lem:local-decomposition} at every
vertex. We now use them to construct the auxiliary directed flow.
For an edge $e\in E$, write $e_L\in L$ and $e_R\in R$ for its two
endpoints. For every incidence $v\in e$, define
\[
\delta_{v,e}
:=
\mathsf J_e^+
-\sum_{\substack{f\ni v\\f\ne e}}\alpha^v_{fe}.
\]
By \cref{eq:local-decomposition}, $\delta_{v,e}\geq0$. Moreover,
\begin{equation}
\label{eq:vertex-imbalance}
\sum_{e\ni v}\delta_{v,e}
=
\mathsf J_v^+-\mathsf J_v^-.
\end{equation}

\begin{definition}
\label{def:auxiliary-network}
The auxiliary directed network $\mc{N}$ has a source $\mathbf s$, a
sink $\mathbf t$, and a node $(v,e,\sigma)$ for every incidence
$v\in e$ and every $\sigma\in\{+,-\}$. Its arcs are of the following
three types.

\begin{itemize}
\item For every $e\in E$, the {traversal arcs} are
\[
(e_L,e,+)\longrightarrow(e_R,e,+)
\qquad\text{and}\qquad
(e_R,e,-)\longrightarrow(e_L,e,-).
\]

\item For every pair of distinct edges $e,f\ni v$, the {transfer arcs} are
\[
(v,e,-)\longrightarrow(v,f,+),
\qquad v\in L,
\]
and
\[
(v,f,+)\longrightarrow(v,e,-),
\qquad v\in R.
\]

\item The {source and sink arcs} are
\[
\mathbf s\longrightarrow(v,e,+),
\qquad v\in L,
\]
and
\[
(v,e,+)\longrightarrow\mathbf t,
\qquad v\in R.
\]
\end{itemize}
\end{definition}

Away from the source and sink, every directed route follows the same
alternating pattern: it traverses a $+$ arc from $L$ to $R$, transfers
to a $-$ node, traverses a $-$ arc from $R$ to $L$, and then transfers
back to a $+$ node.

We now assign flow to the arcs of $\mc{N}$. The traversal arc
associated with $e$ and directed from $L$ to $R$ carries
$\mathsf J_e^+$, while the traversal arc directed from $R$ to $L$
carries $\mathsf J_e^-$. Each transfer arc indexed by
$(v,e,f)$ carries $\alpha^v_{ef}$. Finally, every source or sink arc
incident to $(v,e,+)$ carries $\delta_{v,e}$. Denote the resulting
nonnegative function on the arcs of $\mc{N}$ by $\varphi$. A typical internal segment of a directed route is shown in
\cref{fig:flow-motif}.

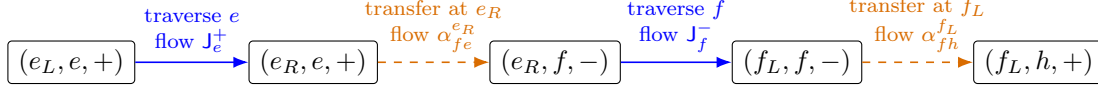
\begin{figure}[t]
\centering
\begin{tikzpicture}[
  state/.style={
    draw,
    rounded corners=1.5pt,
    inner xsep=4pt,
    inner ysep=3pt,
    font=\small
  },
  traversal/.style={
    -{Latex[length=2mm]},
    semithick,
    blue
  },
  transfer/.style={
    -{Latex[length=2mm]},
    semithick,
    dashed,
    orange!85!black
  },
  lab/.style={
    font=\scriptsize,
    align=center
  }
]
\node[state] (a) at (0,0) {$(e_L,e,+)$};
\node[state] (b) at (3.2,0) {$(e_R,e,+)$};
\node[state] (c) at (6.4,0) {$(e_R,f,-)$};
\node[state] (d) at (9.6,0) {$(f_L,f,-)$};
\node[state] (q) at (12.8,0) {$(f_L,h,+)$};

\draw[traversal] (a) -- node[lab,above]
  {traverse $e$\\flow $\mathsf J_e^+$} (b);
\draw[transfer] (b) -- node[lab,above]
  {transfer at $e_R$\\flow $\alpha^{e_R}_{fe}$} (c);
\draw[traversal] (c) -- node[lab,above]
  {traverse $f$\\flow $\mathsf J_f^-$} (d);
\draw[transfer] (d) -- node[lab,above]
  {transfer at $f_L$\\flow $\alpha^{f_L}_{fh}$} (q);
\end{tikzpicture}
\caption{A segment of the auxiliary network $\mc{N}$, with the value
of the flow $\varphi$ indicated on each arc. Solid traversal arcs move
along edges of $G$, while dashed transfer arcs switch to a distinct
incident edge at the current vertex.}
\label{fig:flow-motif}
\end{figure}

\begin{lemma}
\label{lem:auxiliary-flow}
The function $\varphi$ is an $\mathbf s$--$\mathbf t$ flow on
$\mc{N}$. Its value is
\begin{equation}
\label{eq:total-path-flow}
\operatorname{val}(\varphi)
=
\sum_{e\in E}(\mathsf J_e^+-\mathsf J_e^-)
=
\frac12\sum_{v\in V}
  (\mathsf J_v^+-\mathsf J_v^-).
\end{equation}
\end{lemma}

\begin{proof}
At a $-$ node $(v,e,-)$, the traversal arc carries
$\mathsf J_e^-$, while the incident transfer arcs carry a total flow
\[
\sum_{\substack{f\ni v\\f\ne e}}\alpha^v_{ef}
=
\mathsf J_e^-
\]
by \cref{eq:local-decomposition}. Thus flow is conserved at every
$-$ node.

At a $+$ node $(v,e,+)$, the incident transfer arcs carry a total
flow
\[
\sum_{\substack{f\ni v\\f\ne e}}\alpha^v_{fe},
\]
and the incident source or sink arc carries $\delta_{v,e}$. Their sum
is $\mathsf J_e^+$ by the definition of $\delta_{v,e}$, which is the
flow on the traversal arc. Thus flow is also conserved at every
$+$ node.

Finally, \cref{eq:vertex-imbalance} gives
\begin{align*}
\sum_{v\in L}\sum_{e\ni v}\delta_{v,e}
&=
\sum_{v\in L}(\mathsf J_v^+-\mathsf J_v^-)
=
\sum_{e\in E}(\mathsf J_e^+-\mathsf J_e^-) =
\sum_{v\in R}(\mathsf J_v^+-\mathsf J_v^-)
=
\sum_{v\in R}\sum_{e\ni v}\delta_{v,e}.
\end{align*}
Hence the total flow leaving $\mathbf s$ equals the total flow
entering $\mathbf t$, and their common value is
\[
\sum_{e\in E}(\mathsf J_e^+-\mathsf J_e^-).
\]
Since each edge has one endpoint in each vertex class,
\[
\sum_{v\in V}(\mathsf J_v^+-\mathsf J_v^-)
=
2\sum_{e\in E}(\mathsf J_e^+-\mathsf J_e^-),
\]
which proves \cref{eq:total-path-flow}.
\end{proof}

\subsection{Using the girth}

We now decompose the auxiliary flow into paths and cycles. Projecting
these components onto $G$ produces nonbacktracking walks, to which we
can apply the girth assumption. Recall that $\star \in L$ is fixed, but otherwise arbitrary.

\begin{proposition}
\label{prop:girth-numerical}
We have
\[
g\mathsf J_\star^+
\leq
\sum_{v\in V}\mathsf J_v^-
+\frac{g-2}{2}
 \sum_{v\in V}(\mathsf J_v^+-\mathsf J_v^-)
+2(\mathsf J_\star^+-\mathsf J_\star^-).
\]
\end{proposition}

\begin{proof}
Let $\varphi$ be the $\mathbf s$--$\mathbf t$ flow on
$\mc{N}$ constructed in \cref{sec:constructing-flow} from the weights
$\mathsf J_e^\pm$ associated with the fixed vertex $\star$. By the
flow decomposition theorem
\cite[Theorem~3.5]{AhujaMagnantiOrlin}, there is a finite collection
$\mc{C}$ of directed $\mathbf s$--$\mathbf t$ paths and directed
cycles in $\mc{N}$, together with weights $\lambda_C\geq0$, such that
for every arc $a$ of $\mc{N}$,
\[
\varphi(a)
=
\sum_{\substack{C\in\mc{C}\\a\in C}}\lambda_C.
\]

Project $C$ onto $G$ by listing its traversal arcs in their directed
order, replacing each by its underlying edge of $G$, and discarding
the transfer, source, and sink arcs. Denote the resulting walk by
$\pi(C)$. Successive edges of $\pi(C)$ are distinct because every
transfer changes the edge index. Thus $\pi(C)$ is nonbacktracking. A
path component projects to a walk from $L$ to $R$, while a cycle
component projects to a nonempty closed walk. Neither walk needs to be
simple.

Every nonempty closed nonbacktracking walk in $G$ has length at least
$g$, since such a walk contains a simple cycle.

\smallskip
\paragraph{\bf Component bounds.}
For each component $C$, let $k(C)$ be the number of $+$ traversal
arcs leaving $\star$, and let $m(C)$ be the total number of $-$
traversal arcs. If $C$ is a path, let $\varepsilon(C)=1$ when its
source arc enters a node over $\star$, and let $\varepsilon(C)=0$
otherwise. We claim that
\begin{align*}
gk(C)
&\leq 2m(C),
&&\text{if $C$ is a cycle},\\
gk(C)
&\leq 2m(C)+(g-2)+2\varepsilon(C),
&&\text{if $C$ is a path}.
\end{align*}

If $k(C)=0$, both inequalities are immediate, so assume that
$k(C)>0$.

Suppose first that $C$ is a cycle in $\mc{N}$. List the $+$
traversal arcs leaving $\star$ as
$a_1,\ldots,a_{k(C)}$ in their cyclic order along $C$. For each $i$,
with indices taken modulo $k(C)$, consider the portion of $C$ that
starts with $a_i$ and ends with the last traversal arc before
$a_{i+1}$. The first traversal in this portion leaves $\star$, while
the last is a $-$ traversal returning to $\star$. Its projection is
therefore a closed nonbacktracking walk in $G$ based at $\star$. These $k(C)$ portions partition the traversal arcs of $C$. Each
projected walk has length at least $g$ and, since the traversal signs
alternate, contains at least $g/2$ traversals of sign $-$. Therefore,
\[
m(C)\geq\frac g2k(C).
\]

Suppose now that $C$ is an $\mathbf s$--$\mathbf t$ path. List the
$+$ traversal arcs leaving $\star$ as
$a_1,\ldots,a_{k(C)}$ in their order along $C$. For each
$i<k(C)$, the portion beginning with $a_i$ and ending with the last
traversal arc before $a_{i+1}$ projects to a closed nonbacktracking
walk in $G$ based at $\star$. These $k(C)-1$ portions therefore
contain at least
\[
\frac g2\bigl(k(C)-1\bigr)
\]
traversals of sign $-$. If $\varepsilon(C)=0$, the path does not begin over $\star$ and must
enter $\star$ along a $-$ traversal before $a_1$. This gives one
additional $-$ traversal outside the portions above. Hence
\[
m(C)
\geq
\frac g2\bigl(k(C)-1\bigr)+1-\varepsilon(C),
\]
which is equivalent to the claimed path inequality.

\smallskip
\paragraph{\bf Reassembling the flow.}
Multiplying the component bounds by $\lambda_C$ and summing over the flow
decomposition gives
\begin{align}
g\sum_C\lambda_Ck(C)
\leq{}&
2\sum_C\lambda_Cm(C)
+(g-2)\sum_{\substack{C\text{ a path}}}\lambda_C
+2\sum_{\substack{C\text{ a path}}}
  \lambda_C\varepsilon(C).
\label{eq:summed-inequality}
\end{align}
We now identify these four quantities in terms of $\varphi$.

The term on the left is the total flow on the $+$ traversal arcs
leaving $\star$, and hence
\[
\sum_C\lambda_Ck(C)=\sum_{e\ni \star}\mathsf J_e^+=\mathsf J_\star^+.
\]
Similarly, $\sum_C\lambda_Cm(C)$ is the total flow on all $-$
traversal arcs. Therefore,
\[
\sum_C\lambda_Cm(C)
=
\sum_{e\in E}\mathsf J_e^-
=
\frac12\sum_{v\in V}\mathsf J_v^-.
\]

The total weight of the path components is the value  of the flow, so
\cref{eq:total-path-flow} gives
\[
\sum_{\substack{C\text{ a path}}}\lambda_C
=\text{val}(\varphi)=
\frac12\sum_{v\in V}
  (\mathsf J_v^+-\mathsf J_v^-).
\]
Finally, $\varepsilon(C)=1$ precisely when the source arc of $C$
enters a node over $\star$. Hence, by
\cref{eq:vertex-imbalance},
\[
\sum_{\substack{C\text{ a path}}}
  \lambda_C\varepsilon(C)
=
\sum_{e\ni\star}\delta_{\star,e}
=
\mathsf J_\star^+-\mathsf J_\star^-.
\]

Substituting these identities into \cref{eq:summed-inequality}
gives
\[
g\mathsf J_\star^+
\leq
\sum_{v\in V}\mathsf J_v^-
+\frac{g-2}{2}
 \sum_{v\in V}(\mathsf J_v^+-\mathsf J_v^-)
+2(\mathsf J_\star^+-\mathsf J_\star^-),
\]
as required.
\end{proof}

We
next translate the preceding proposition into the change in $\Phi_0$ produced by
conditioning on $X_\star$.

\begin{proof}[Proof of \cref{prop:girth-estimate}]
By \cref{eq:Phi0-def} and \cref{eq:Jv-information},
\[
g\left(
  \Phi_0(\bs{p})
  -\E_{X_\star}\Phi_0(\bs{p}^{X_\star})
\right)
=
\sum_{v\in V}\mathsf J_v^-
+\frac{g-2}{2}
 \sum_{v\in V}(\mathsf J_v^+-\mathsf J_v^-).
\]

At the distinguished vertex $\star \in L$, $\bs{p}_\star$ is the law of
$X_\star$, while $\bs{p}_\star^{X_\star}$ is a point mass. Hence
\[
\mathsf J_\star^+=H(X_\star),
\qquad
\mathsf J_\star^+-\mathsf J_\star^-
=
S(\bs{p}_\star).
\]
Substituting these identities into \cref{prop:girth-numerical} gives
\[
gH(X_\star)
\leq
g\left(
  \Phi_0(\bs{p})
  -\E_{X_\star}\Phi_0(\bs{p}^{X_\star})
\right)
+2S(\bs{p}_\star).
\]
Dividing by $g$ and rearranging proves
\[
\Phi_0(\bs{p})
-\E_{X_\star}\Phi_0(\bs{p}^{X_\star})
-H(X_\star)
\geq
-\frac2gS(\bs{p}_\star).
\]
Since $\star\in L$ is arbitrary, this proves
\cref{prop:girth-estimate}.
\end{proof}
\section{The permutation estimate}
\label{sec:permutation}

We now prove \cref{prop:averaging-estimate}. Until \cref{sec:matching-permutation}, the graph plays no role and the argument applies to an
arbitrary random permutation. 

Recall from \cref{eq:Phi-decomposition} that
\[
  \Gamma(\bs{q})
  =H(\bs{q})-2H_{\mr c}(\bs{q})+2F(\bs{q}).
\]

\subsection{Deletion and concavity}

We first establish a recursion for $\Gamma$. We use it both
to prove the concavity of $\Gamma$ and to establish
\cref{lem:avg-over-cond-coords}.

For $\bs{q}\in\Delta_d$ and $a\in[d]$ with $q_a<1$, let
\[
  \bs{q}^{\setminus a}
  :=
  \left(\frac{q_i}{1-q_a}\right)_{i\ne a}
  \in\Delta_{d-1}.
\]
When $q_a=1$, an expression of the form
\[
(1-q_a)\Gamma(\bs{q}^{\setminus a})
\]
is interpreted as zero. Thus, the right-hand side of
\cref{eq:delete-one} is well defined also when $\bs{q}$ is a point
mass.

\begin{lemma}
\label{lem:delete-one-identity}
For every $\bs{q}\in\Delta_d$,
\begin{equation}\label{eq:delete-one}
  d\Gamma(\bs{q})
  =
  S(\bs{q})
  +\sum_{a=1}^d(1-q_a)\Gamma(\bs{q}^{\setminus a}).
\end{equation}
\end{lemma}

\begin{proof}
If $\bs{q}$ is a point mass, then both sides of
\cref{eq:delete-one} vanish. We may therefore assume that $q_a<1$
for every $a$. Let
\[
  R(\bs{q})
  :=F(\bs{q})+S(\bs{q})
  =
  \E_\pi\left[
    \sum_{t=1}^d q_{\pi(t)}
    \log\frac{s_t^\pi(\bs{q})}{q_{\pi(t)}}
  \right],
\]
so that $\Gamma=2R-H$. Condition on the first element
$a=\pi(1)$ of the random ordering. This element contributes
$q_a\log(1/q_a)$ to $R(\bs{q})$. After deleting $a$ and renormalizing,
the remaining contribution is
$(1-q_a)R(\bs{q}^{\setminus a})$. Averaging over the $d$ possible
first elements gives
\begin{equation}\label{eq:R-recursion}
  dR(\bs{q})
  =
  H(\bs{q})
  +\sum_{a=1}^d(1-q_a)R(\bs{q}^{\setminus a}).
\end{equation}

Note that ordinary entropy satisfies the analogous identity
\begin{equation}\label{eq:H-delete}
  \sum_{a=1}^d(1-q_a)H(\bs{q}^{\setminus a})
  =
  (d-1)H(\bs{q})-H_{\mr c}(\bs{q}).
\end{equation}
Indeed, let $J\sim\bs{q}$, choose $A$ uniformly from $[d]$
independently of $J$, and let $B=\mbf{1}_{\{J\ne A\}}$. Conditioning
first on $(A,B)$ gives
\[
  H(J\mid A,B)
  =
  \frac1d\sum_{a=1}^d
  (1-q_a)H(\bs{q}^{\setminus a}).
\]
On the other hand, since $B$ is determined by $(J,A)$,
\[
  H(J\mid A,B)
  =
  H(J\mid A)-H(B\mid A)
  =
  H(\bs{q})-\frac1d\sum_{a=1}^d h(q_a).
\]
Equating these expressions and using
\[
  \sum_{a=1}^d h(q_a)
  =
  H(\bs{q})+H_{\mr c}(\bs{q})
\]
proves \cref{eq:H-delete}. Finally, substituting $\Gamma=2R-H$ into
\cref{eq:R-recursion} and applying \cref{eq:H-delete} gives \cref{eq:delete-one}.
\end{proof}

The deletion identity has the following information-theoretic
consequence. Let $J\sim\bs{q}\in\Delta_d$, where $d\geq2$, and, conditioned on
$J$, let $W$ be uniformly distributed on $[d]\setminus\{J\}$. Then
\[
  \P(W=a)=\frac{1-q_a}{d-1},
  \qquad
  \mc{L}(J\mid W=a)=\bs{q}^{\setminus a}
\]
whenever $q_a<1$. By the definition of $\Gamma$-information gain,
\begin{align*}
  \mathsf I_\Gamma(J;W)
  &=
  \Gamma\bigl(\mc{L}(J)\bigr)
  -\E_W\Gamma\bigl(\mc{L}(J\mid W)\bigr)\\
  &=
  \Gamma(\bs{q})
  -\sum_{a=1}^d
    \P(W=a)\,
    \Gamma\bigl(\mc{L}(J\mid W=a)\bigr)\\
  &=
  \Gamma(\bs{q})
  -\frac1{d-1}\sum_{a=1}^d
    (1-q_a)\Gamma(\bs{q}^{\setminus a}),
\end{align*}
where a term with $q_a=1$ is interpreted as zero. Hence
\[
  (d-1)\mathsf I_\Gamma(J;W)
  =
  (d-1)\Gamma(\bs{q})
  -\sum_{a=1}^d
    (1-q_a)\Gamma(\bs{q}^{\setminus a}).
\]
Similarly,
\[
  \mathsf I_\Gamma(J;J)=\Gamma(\bs{q}),
\]
because conditioning on $J$ produces point-mass distributions and
$\Gamma$ vanishes at point masses. Therefore,
\begin{equation}
\label{eq:delete-calibration}
  \mathsf I_\Gamma(J;J)
  +(d-1)\mathsf I_\Gamma(J;W)
  =
  d\Gamma(\bs{q})
  -\sum_{a=1}^d
    (1-q_a)\Gamma(\bs{q}^{\setminus a}) =S(\bs{q}),
\end{equation}
where the final equality is \cref{eq:delete-one}.

\begin{lemma}\label{lem:Gamma-concavity}
For every $d\geq1$, the functional $\Gamma:\Delta_d\to\R$ is
concave.
\end{lemma}

\begin{proof}
Let $\Gamma_d$ denote the restriction of $\Gamma$ to $\Delta_d$. We
argue by induction on $d$. Since $\Gamma_1=0$, the claim is immediate
for $d=1$. Suppose that $\Gamma_{d-1}$ is concave. For
$\bs{q}\in\Delta_d$, let
 $ \bs{q}_{-a}:=(q_i)_{i\ne a}.
$
By \cref{eq:delete-one},
\[
  \Gamma_d(\bs{q})
  =
  \frac1dS(\bs{q})
  +\frac1d\sum_{a=1}^d
  (1-q_a)\,
  \Gamma_{d-1}\left(\frac{\bs{q}_{-a}}{1-q_a}\right),
\]
where a summand with $q_a=1$ is interpreted as zero. For each $a$,
the map
\[
  \bs{q}
  \longmapsto
  (1-q_a)\,
  \Gamma_{d-1}\left(\frac{\bs{q}_{-a}}{1-q_a}\right)
\]
is the perspective of $\Gamma_{d-1}$, composed with the affine map
$\bs{q}\mapsto(1-q_a,\bs{q}_{-a})$, and is therefore concave; see
\cite[Section~3.2.6]{BoydVandenberghe}. Since $S$ is concave by
\cref{lem:b-properties}, it follows that
$\Gamma_d$ is concave.
\end{proof}

\subsection{An information inequality for random permutations}

We now show that, for any fixed coordinate $X_j$, the sum over $i$ of
the $\Gamma$-information gained about $X_j$ by observing $X_i$ is at
least the local Bethe entropy of the law of $X_j$.
\begin{lemma}
\label{lem:avg-over-cond-coords}
Let $X=(X_1,\ldots,X_d)$ be a random permutation of $[d]$, not
necessarily uniformly distributed. Then, for every $j\in[d]$,
\[
  \sum_{i=1}^d\mathsf I_\Gamma(X_j;X_i)
  \geq S\bigl(\mc{L}(X_j)\bigr).
\]
\end{lemma}
\begin{proof}
The case $d=1$ is immediate. Suppose that $d\geq2$, fix $j\in[d]$,
and let $\bs{q}:=\mc{L}(X_j)$. Choose $K$ uniformly from
$[d]\setminus\{j\}$, independently of $X$, and set
\[
  T:=(K,X_K),
  \qquad
  W:=X_K.
\]
Thus, $T$ records both the chosen coordinate and its value, while
$W$ records only the value. Since $K$ is independent and uniform,
\[
  \mathsf I_\Gamma(X_j;T)
  =
  \frac1{d-1}
  \sum_{i\ne j}\mathsf I_\Gamma(X_j;X_i).
\]

Conditioned on $X_j$, the values in the remaining coordinates are
exactly the elements of $[d]\setminus\{X_j\}$. Since $K$ is uniform
over these coordinates, $W=X_K$ is uniform over
$[d]\setminus\{X_j\}$. Therefore, \cref{eq:delete-calibration} gives
\[
  \mathsf I_\Gamma(X_j;X_j)
  +(d-1)\mathsf I_\Gamma(X_j;W)
  =
  S(\bs{q}).
\]

On the other hand, $W$ is obtained from $T=(K,X_K)$ by discarding
the index $K$. Thus, $X_j$--$T$--$W$ is a Markov chain. Since $\Gamma$ is
concave by \cref{lem:Gamma-concavity}, the data-processing inequality
in \cref{lem:data-processing} gives
\[
  \mathsf I_\Gamma(X_j;T)
  \geq
  \mathsf I_\Gamma(X_j;W).
\]
It follows that
\begin{align*}
  \sum_{i=1}^d\mathsf I_\Gamma(X_j;X_i)
  &=
  \mathsf I_\Gamma(X_j;X_j)
  +(d-1)\mathsf I_\Gamma(X_j;T)\\
  &\geq
  \mathsf I_\Gamma(X_j;X_j)
  +(d-1)\mathsf I_\Gamma(X_j;W)\\
  &=S(\bs{q}). \qedhere
\end{align*}
\end{proof}

\subsection{Completing the proof}
\label{sec:matching-permutation}

We now deduce \cref{prop:averaging-estimate} from
\cref{lem:avg-over-cond-coords}. Fix identifications of $L$ and $R$
with $[n]$. Under these identifications, the matching bijection
$X:L\to R$ and its inverse are random permutations of $[n]$.

For a vertex $v$, the law of its matched neighbor is $\bs{p}_v$
after adjoining zero coordinates corresponding to nonneighbors.
The same is true of every conditional law obtained by observing a
coordinate of $X$. By \cref{rem:zero-coordinates}, adjoining these
zero coordinates does not affect $S$, $F$, or $\Gamma$, and hence
does not affect the corresponding $\Gamma$-information gains.

For a left vertex $v$, \cref{lem:avg-over-cond-coords} therefore
applies directly to $Z_v=X_v$. For a right vertex $w$, applying the
lemma to $X^{-1}$ controls the $\Gamma$-information about
$(X^{-1})_w$ provided by the inverse coordinates $(X^{-1})_y$. The
proposition instead involves the information provided by the forward
coordinates $X_u$. Although the individual terms need not agree,
their sums do.

\begin{lemma}
\label{lem:forward-inverse}
For every $w\in R$,
\[
  \sum_{u\in L}\mathsf I_\Gamma((X^{-1})_w;X_u)
  =
  \sum_{y\in R}\mathsf I_\Gamma((X^{-1})_w;(X^{-1})_y).
\]
\end{lemma}

\begin{proof}
Fix $w\in R$ and write $Z:=(X^{-1})_w$. Expanding the definition of
$\Gamma$-information gain gives
\begin{align*}
  \sum_{u\in L}\mathsf I_\Gamma(Z;X_u)
  &=
  n\Gamma\bigl(\mc{L}(Z)\bigr)
  -\sum_{\substack{u\in L\\y\in R}}
    \P(X_u=y)\,
    \Gamma\bigl(\mc{L}(Z\mid X_u=y)\bigr),\\
  \sum_{y\in R}\mathsf I_\Gamma(Z;(X^{-1})_y)
  &=
  n\Gamma\bigl(\mc{L}(Z)\bigr)
  -\sum_{\substack{y\in R\\u\in L}}
    \P((X^{-1})_y=u)\,
    \Gamma\bigl(\mc{L}(Z\mid (X^{-1})_y=u)\bigr).
\end{align*}
The two double sums are equal because
\[
  \{X_u=y\}=\{(X^{-1})_y=u\}. \qedhere
\]
\end{proof}

\begin{proof}[Proof of \cref{prop:averaging-estimate}]
For each $v\in L$, we have $Z_v=X_v$, so
\cref{lem:avg-over-cond-coords} gives
\[
  \sum_{u\in L}\mathsf I_\Gamma(Z_v;X_u)
  \geq
  S\bigl(\mc{L}(X_v)\bigr)
  =
  S(\bs{p}_v).
\]
For each $w\in R$, we have $Z_w=(X^{-1})_w$. Applying
\cref{lem:avg-over-cond-coords} to $X^{-1}$ and using
\cref{lem:forward-inverse}, we obtain
\[
  \sum_{u\in L}\mathsf I_\Gamma(Z_w;X_u)
  =
  \sum_{y\in R}\mathsf I_\Gamma((X^{-1})_w;(X^{-1})_y)
  \geq
  S\bigl(\mc{L}((X^{-1})_w)\bigr)
  =
  S(\bs{p}_w).
\]
Summing over the two vertex classes yields
\begin{align*}
  \sum_{u\in L}\sum_{v\in V}
  \mathsf I_\Gamma(Z_v;X_u)
  &\geq
  \sum_{v\in V}S(\bs{p}_v)
  =2H_{\mr B}(\bs{p})
  =2\sum_{u\in L}S(\bs{p}_u),
\end{align*}
where \cref{eq:bethe-entropy-local} gives
\[
  \sum_{v\in L}S(\bs{p}_v)
  =
  \sum_{v\in R}S(\bs{p}_v)
  =
  H_{\mr B}(\bs{p}).
\]
\end{proof}

\bibliographystyle{amsplain0}
\bibliography{main}

\end{document}